\documentclass[12pt]{article}
\usepackage{geometry}                
\usepackage{amsmath, amsthm, amssymb}
\newcommand{\R}{\mathbb{R}}

\newcommand{\F}{\mathcal{F}}

\newcommand{\N}{\mathbb{N}}
\theoremstyle{definition}
\newtheorem{definition}{Definition} \numberwithin{definition}{section}

\newtheorem{remark}[definition]{Remark}
\newtheorem{notation}[definition]{Notation}
\newtheorem{sassumption}[definition]{Standing assumption}

\theoremstyle{plain}
\newtheorem{theorem}[definition]{Theorem}
\newtheorem{lemma}[definition]{Lemma}
\newtheorem{proposition}[definition]{Proposition}
\newtheorem{corollary}[definition]{Corollary}

\title{A note on asymptotic exponential arbitrage with exponentially decaying failure probability}
\author{Kai Du\footnotemark[1] \and Ariel David Neufeld\footnotemark[1]}

\begin{document}
\maketitle
\pagenumbering{Roman}
\pagenumbering{arabic}

\begin{abstract}
The goal of this paper is to prove a result conjectured in F{\"o}llmer and
Schachermayer \cite{MR2403770}, even in slightly more general form. Suppose
that $S$ is a continuous semimartingale and satisfies a large deviations
estimate; this is a particular growth condition on the mean-variance tradeoff
process of $S$. We show that $S$ then allows asymptotic exponential arbitrage
with exponentially decaying failure probability, which is a strong and
quantitative form of long-term arbitrage. In contrast to F{\"o}llmer and
Schachermayer \cite{MR2403770}, our result does not assume that $S$ is a
diffusion, nor does it need any ergodicity assumption. \bigskip

\noindent\textbf{Keywords:} Asymptotic exponential arbitrage; large
deviations; continuous semimartingale model.\smallskip

\noindent\textbf{AMS 2010 Subject Classifications:} 91G10, 60F10, 60G44.

\end{abstract}

\footnotetext[1]{Department of Mathematics, ETH Z\"{u}rich, 8092 Z\"{u}rich,
Switzerland. Email addresses: \texttt{kai.du@math.ethz.ch} (K. Du),
\texttt{ariel.neufeld@math.ethz.ch} (A. D. Neufeld).}
\section{Introduction}
Let $(\Omega, \F,\mathbb{F},P)$ be a filtered probability space where the filtration $\mathbb{F}=(\F_t)_{t\geq 0}$
satisfies the usual conditions and let the price process $S=(S_t)_{t\geq 0}$ initially be any $\R^d$-valued semimartingale.
We define for each $T>0$ the set
\begin{equation*}
\mathbf{K}^T:= \bigg\{ \int_0^T H_s\, dS_s \ \bigg| \ H \in L(S) \
\mbox{admissible, i.e.} \ \int H\, dS \geq -a  \ \ \mbox{for some} \ a \in
\R_+ \bigg\}.
\end{equation*}
The following form of a long-term arbitrage was considered for the first time in F{\"o}llmer and Schachermayer \cite{MR2403770}; its name is taken from Mbele Bidima and R\'asonyi \cite{MbeleR}.
\begin{definition}
The process $S=(S_t)_{t\geq 0}$ \textit{allows asymptotic exponential arbitrage with exponentially decaying failure probability} if there exist $0<\tilde{T}<\infty$ and constants $C, \gamma_1, \gamma_2 >0 $ such that for all $T\geq \tilde{T}$, there is $X_T \in \mathbf{K}^T$ with\\
$a)$ $X_T \geq -e^{-\gamma_1 T} \ \ P$-a.s.\\
$b)$ $P[X_T \leq e^{\gamma_1 T}] \leq C e^{-\gamma_2 T}$ .
\end{definition}

If $S$ has that property, we can find for any large enough maturity $T$, up to an exponentially
(in $T$) small probability of failure, an exponentially (in $T$) large profit with an exponentially
(in $T$) small potential loss. This gives an explicit relation between any tolerance level of failure and the necessary time
to reach a high level. Furthermore, when $T \to \infty$, we get in the limit a riskless profit. Thus, asymptotic exponential arbitrage with exponentially decaying failure probability can be interpreted as a strong and quantitative form of long-term arbitrage.

We define the sets
\begin{equation*}
\mathbf{M}^{T,e}_m :=\Big\{ Q \ \mbox{p.m. on} \ \F_T \ \Big| \ Q \approx
P|_{\F_T} \ \mbox{and} \ (S_t)_{0\leq t \leq T} \ \mbox{ is a local
$Q$-martingale} \Big\}.
\end{equation*}
\begin{sassumption}\label{assAEAgenc}
Throughout this paper, we assume that $\mathbf{M}^{T,e}_m \neq \emptyset$ for
any $0<T<\infty$ and that the filtration $\mathbb{F}$ is \textit{continuous},
i.e. every local martingale with respect to $\mathbb{F}$ is continuous.
\end{sassumption}
We show below that under Assumption \ref{assAEAgenc}, any semimartingale in $\mathbb{F}$ is in fact continuous. Moreover, using a result of Schweizer \cite{Schweizer1995},
 we show in Lemma \ref{AEAchardensgenc} that there exists a predictable, sufficiently integrable $\R^d$-valued process $\lambda=(\lambda_t)_{t\geq 0}$ such that for any $T<\infty$ and any $Q \in \mathbf{M}^{T,e}_m$, the density process $Z^Q=(Z^Q_t)_{0\leq t \leq T}$ of $Q$ with respect to $P|_{\F_T}$ is of the form
\begin{equation*}
Z^Q= Z_0^Q \,\mathcal{E}\bigg(\int -\lambda \, dM + N^Q\bigg)=: Z_0^Q \, \mathcal{E}\big(L^Q\big) \ \ \mbox{on} \ [\![0,T]\!] ,
\end{equation*}
where $N^Q=(N^Q_t)_{0\leq t \leq T}$ is a continuous local martingale with $N^Q \bot M^T$ and $M^T$ is
the continuous local martingale coming from the canonical decomposition of $S^T$.
We call $\lambda$ a \textit{market price of risk} for the price process $S$.

Following F{\"o}llmer and Schachermayer \cite{MR2403770}, we extend the notion of $S$ satisfying a large deviations estimate.
\begin{definition}\label{large6}
A market price of risk $\lambda=(\lambda_t)_{t\geq 0}$ for the price process $S=(S_t)_{t\geq 0}$ \textit{satisfies a large deviations estimate} if there exist constants $c_1,c_2 >0$ such that
\begin{equation} \label{large6a}
\limsup_{T\to \infty} \frac{1}{T} \log P\bigg[\frac{1}{T} \int_0^T
\lambda^{\mathrm{tr}}_s \, d\langle M \rangle_s \, \lambda_s \leq c_1
\bigg]<-c_2 .
\end{equation}
\\
\end{definition}
The main goal of this paper is to prove that under Assumption \ref{assAEAgenc}, if a market price of risk for the price process $S$ satisfies a large deviations estimate, then $S$ allows asymptotic exponential arbitrage with exponentially decaying failure probability.

In F{\"o}llmer and Schachermayer \cite{MR2403770}, the authors considered an $\R^d$-valued diffusion process
$\tilde{S}=(\tilde{S}_t)_{t\geq 0}$ defined over a filtered probability space
$(\tilde{\Omega}, \tilde{\F},\tilde{\mathbb{F}},\tilde{P})$, where the filtration
$\tilde{\mathbb{F}}=(\tilde{\F}_t)_{t\geq 0}$ is the $\tilde{P}$-augmentation of the raw filtration generated by an
$\R^N$-valued Brownian motion $\tilde{W}$ and $\tilde{S}$ is of the form
\begin{equation}\label{foschaN6}
d\tilde{S}_t=\sigma(\tilde{S}_t) \big(d\tilde{W}_t + \varphi(\tilde{S}_t)\, dt\big) .
\end{equation}
In (\ref{foschaN6}), $\sigma: \R^d \rightarrow \R^{d\times N}$ and $\varphi: \R^d \rightarrow \R^{N}$
are such that $\varphi(\tilde{S}_t) \in (\ker(\sigma(\tilde{S}_t))^{\bot}$ for any $t \geq 0$ and
the process $\tilde{Z}=(\tilde{Z}_t)_{t\geq 0}$ defined by
\begin{equation}\label{AEAfoschamartass}
\tilde{Z}_t:= \mathcal{E}\bigg(-\int \varphi(\tilde{S}) \, d\tilde{W}\bigg)_t = \exp \bigg(-\int_0^t \varphi(\tilde{S}_s)\, d\tilde{W}_s - \frac{1}{2} \int_0^t \Vert \varphi(\tilde{S}_s) \Vert^2\, ds \bigg)
\end{equation}
is a strictly positive $\tilde{P}$-martingale, where $\Vert \cdot \Vert$ denotes the Euclidean norm on $\R^N$.
%
\begin{definition}\label{large6diff}
The market price of risk function $\varphi(\cdot)$ for the price process $\tilde{S}$ \textit{satisfies a large deviations estimate with respect to} $\tilde{S}$ if there are constants $c_1,c_2 >0$ such that
\begin{equation}\label{large6aa}
\limsup_{T\to \infty} \frac{1}{T} \log P\bigg[\frac{1}{T} \int_0^T \Vert \varphi(\tilde{S}_s) \Vert^2\, ds \leq c_1 \bigg]<-c_2 .
\end{equation}
\end{definition}
%

%

F{\"o}llmer and Schachermayer formulated in \cite{MR2403770} the conjecture that if (\ref{large6aa}) holds,
then $\tilde{S}$ allows asymptotic exponential arbitrage with exponentially decaying failure probability.
In Mbele Bidima and R\'asonyi \cite{MbeleR}, the authors proved such a result in a discrete-time version
of the model (\ref{foschaN6}) with bounded drift and volatility.
In the present paper, we can show, as a corollary of our main theorem,
that the conjecture is also true in the stated form for the continuous-time price process $\tilde{S}$ in (\ref{foschaN6}).
\section{Main theorem, its proof and comments}
We begin by showing
\begin{lemma}\label{AEAgencont}
Under Assumption \ref{assAEAgenc}, the process $S$ is continuous.
\end{lemma}
\begin{proof}
This is well known, but we give a proof for completeness. Take any
decomposition $S = S_0 + M + A$ with a local martingale $M$ and with $A$ of
finite variation. Take any $T<\infty$ and any $Q \in \mathbf{M}^{T,e}_m$ with
density process $Z=(Z_t)_{0\leq t \leq T}$ with respect to $P|_{\F_T}$. As
$\mathbb{F}$ is continuous, we obtain that the local $P$-martingales $Z$,
$ZS$ and $M$ and hence also $\frac{1}{Z}$ are continuous processes up to time
$T$. So $A=\frac{1}{Z} ZS - M - S_0$ and hence also $S$ are continuous up to
time $T$, which gives the result.
\end{proof}
\begin{notation} Under Assumption \ref{assAEAgenc}, we let
\begin{equation*}
S=S_0 + M + A
\end{equation*}
be the canonical decomposition of the continuous semimartingale $S$, where $M$ is a continuous local martingale and $A$ is a continuous process of finite variation.
\end{notation}
We next characterize for any $Q \in \mathbf{M}^{T,e}_m$ the structure of its
density process $Z^Q$ with respect to $P|_{\F_T}$. For unexplained notations
from martingale theory, we refer to Jacod and Shiryaev \cite{MR1943877}.
\begin{lemma}\label{AEAchardensgenc}
Under Assumption \ref{assAEAgenc}, there exists an $\R^d$-valued stochastic
process $\lambda=(\lambda_t)_{t\geq 0} \in L^2_{\mathrm{loc}}(M)$ such that
for any $T<\infty$ and any $Q \in \mathbf{M}^{T,e}_m$, the density process
$Z^Q=(Z^Q_t)_{0\leq t \leq T}$ of $Q$ with respect to $P|_{\F_T}$ is of the
form
\begin{equation}\label{maegeaea1}
Z^Q= Z_0^Q \, \mathcal{E}\bigg(\int -\lambda\, dM + N^Q\bigg)=: Z_0^Q \, \mathcal{E}\big(L^Q\big) \ \ \mbox{on} \ [\![0,T]\!] ,
\end{equation}
where $N^Q= (N^Q_t)_{0\leq t \leq T}$ is a continuous local martingale with $N^Q \bot M^T$. As a consequence, we have
\begin{equation}\label{maegeaea2}
\big\langle L^Q \big\rangle_t \geq \int_0^t \lambda^{\mathrm{tr}}_s \,
d\langle M\rangle_s \, \lambda_s
\end{equation}
for each $t \in [0,T]$. We call $\lambda$ a \emph{market price of risk} for the price process $S$.
\end{lemma}
\begin{proof}
By Lemma \ref{AEAgencont}, the price process $S$ is continuous. Since
$\mathbf{M}^{T,e}_m \neq \emptyset$ for any $T<\infty$, Theorem 1 in
Schweizer \cite{Schweizer1995} gives for any $T<\infty$ an $\R^d$-valued
process $\lambda^{(T)}=(\lambda^{(T)}_t)_{0\leq t \leq T} \in
L^2_{\mathrm{loc}}(M^T)$ such that for any $Q \in \mathbf{M}^{T,e}_m$, the
density process $Z^Q=(Z^Q_t)_{0\leq t \leq T}$ of $Q$ with respect to
$P|_{\F_T}$ is of the form
\begin{equation}\label{AEAchardensl0genc}
Z^Q= Z_0^Q \, \mathcal{E}\bigg(\int -\lambda^{(T)}\, dM^T + N^Q\bigg) \ \ \mbox{on} \ [\![0,T]\!] ,
\end{equation}
where $N^Q=(N^Q_t)_{0\leq t \leq T}$ is a continuous local martingale with
$N^Q \bot M^T$. We point out that the process $\lambda^{(T)}$ need not be
unique. However, the stochastic integral $\int \lambda^{(T)} dM^T$ does not
depend on the choice of $\lambda^{(T)}$ satisfying (\ref{AEAchardensl0genc});
see Schweizer \cite{Schweizer1995}. Extending $\lambda^{(T)}$ to $[0,\infty)$
by setting $\bar{\lambda}^{(T)}= \lambda^{(T)} 1_{[[0,T]]}$, we clearly have
$\bar{\lambda}^{(T)} \in L^2_{\mathrm{loc}}(M)$. The $\R^d$-valued process
$\lambda=(\lambda_t)_{t\geq 0}$ defined by
\begin{equation}\label{AEAkleben}
\lambda := \sum_{n=1}^{\infty} \bar{\lambda}^{(n)}\, 1_{((n-1,n]]}
\end{equation}
 is then in $L^2_{\mathrm{loc}}(M)$, too. Moreover, $Q |_{\F_{n-1}} \in \mathbf{M}^{n-1,e}_m$ for any $Q \in \mathbf{M}^{n,e}_m$,
and so (\ref{AEAchardensl0genc}) yields inductively that
\begin{equation*}
\int \lambda\, dM = \int \bar{\lambda}^{(n)}\, dM = \int \bar{\lambda}^{(n)}\, dM^{n} \ \ \mbox{on} \ [\![0,n]\!]
\end{equation*}
for any $n \in \N$. So (\ref{maegeaea1}) follows from (\ref{AEAchardensl0genc}) and (\ref{AEAkleben}).
%

Finally, $L^Q = - \int \lambda dM + N^Q$ on $[\![0,T]\!]$ and $N^Q \bot M^T$ imply (\ref{maegeaea2}) because
\begin{equation*}
\big\langle L^Q \big\rangle = \int \lambda^{\mathrm{tr}}\, d\langle M
\rangle\, \lambda + \langle N^Q\rangle .
\end{equation*}
\end{proof}
\begin{remark} We do not claim that the market price of risk $\lambda$ for the price process $S$ is unique.
However, as already used, the stochastic integral $\int \lambda dM$ does not
depend on the choice of $\lambda$. This can for instance be seen by writing
for $Q \in \mathbf{M}^{T,e}_m$ the density process $Z^Q= Z^Q_0
\mathcal{E}(L^Q)$ and then
 arguing that $-\int \lambda dM$ must be the projection of $L^Q$ on $M$; this follows because $Z^Q S$ is a local $P$-martingale.
 As a consequence, the property of satisfying a large deviations estimate does not depend on
 the choice of the market price of risk $\lambda$ either.
\end{remark}
\begin{notation} For brevity, we introduce the so-called \textit{mean-variance tradeoff process}
\begin{equation*}
K_t:= \int_0^t \lambda^{\mathrm{tr}}_s\, d\langle M \rangle_s\, \lambda_s
\end{equation*}
for $t\geq 0$. This process is finite-valued since $\lambda \in
L^2_{\mathrm{loc}}(M)$, and it does not depend on
 the choice of the market price of risk $\lambda$; in fact $K=\langle \int \lambda\, dM \rangle .$
\end{notation}
\begin{lemma}\label{AEAkail0}
Under Assumption \ref{assAEAgenc}, suppose that a market price of risk $\lambda$ for the price process $S$ satisfies a large deviations estimate. Then
\begin{equation*}
K_{\infty} := \lim\limits_{t \to \infty} \  K_t = \infty \ \ P\mbox{-a.s.}
\end{equation*}
\end{lemma}
\begin{proof}
If the above statement is not true, there is a constant $C>0$ with
\begin{equation*}
P\big[ K_{\infty} \leq C \big] =: P[B]>0.
\end{equation*}
As $\lambda$ satisfies a large deviations estimate, there exist constants $c_1,c_2 >0$ such that
\begin{equation*}
\limsup_{T\to \infty} \frac{1}{T} \log P\bigg[\frac{1}{T} \ K_T \leq c_1 \bigg]
 =: \limsup_{T\to \infty} \frac{1}{T} \log P[V_T] < -c_2 .
\end{equation*}
Thus, we can find $0<\bar{T}<\infty$ such that
\begin{equation*}
C \leq c_1 \bar{T} , \ \ \ \ \ \ P[V_{\bar{T}}] \leq e^{-\frac{c_2}{2} \bar{T}} , \ \ \ \ \ \ e^{-\frac{c_2}{2} \bar{T}}<P[B].
\end{equation*}
As $C \leq c_1 \bar{T}$ and $K$ is increasing, we get $B\subseteq V_{\bar{T}}$. But then, by the definition of $\bar{T}$,
\begin{equation*}
P[B]\leq P[V_{\bar{T}}] \leq e^{-\frac{c_2}{2} \bar{T}} < P[B]
\end{equation*}
which gives a contradiction.
\end{proof}
\begin{lemma}\label{AEAkail1}
Under Assumption \ref{assAEAgenc}, suppose that a market price of risk $\lambda$ for the price process $S$ satisfies
 a large deviations estimate. Fix $0<T<\infty$ and let $L=(L_t)_{0\leq t \leq T}$ be a continuous local martingale with $L_0=0$.
Then there exists a continuous local martingale $\bar{L}=(\bar{L}_t)_{t\geq 0}$ such that $\bar{L}_t =L_t$ for any $t \in [0,T]$ and
\begin{equation*}
\big\langle \bar{L} \big\rangle_{\infty} :=\lim\limits_{t \to \infty} \big\langle \bar{L} \big\rangle_{t}= \infty \ \ \ P\mbox{-a.s.}
\end{equation*}
 \end{lemma}
\begin{proof}
Define the process $Y=(Y_t)_{t\geq T}$ by $Y_t:= \int_T^t \lambda_s\, dM_s$, set
$\bar{Y}:= Y 1_{[[T,\infty))}$ and
\begin{equation*}
\bar{L}:= L\, 1_{[[0,T]]} + \bar{Y} = L\, 1_{[[0,T]]} + Y\, 1_{[[T,\infty))}.
\end{equation*}
Then $\bar{L}$ is a continuous local martingale null at 0 like
 $L$, $Y$ and $\bar{Y}$, and we have $\bar{L}=L$ on $[\![0,T]\!]$ by construction. Moreover,
\begin{equation*}
\langle \bar{L} \rangle_{\infty} = \langle L \rangle_T + \langle \bar{Y}
\rangle_{\infty} = \langle L \rangle_T + \int_T^\infty
\lambda^{\mathrm{tr}}_s\, d\langle M \rangle_s\, \lambda_s = \langle L
\rangle_T +K_{\infty} -K_T = \infty \ \ P\mbox{-a.s.}
\end{equation*}
due to Lemma \ref{AEAkail0}.
\end{proof}
\begin{remark} In Lemma \ref{AEAkail1}, we can replace Assumption \ref{assAEAgenc} and the condition on $\lambda$
by assuming instead that there exists a Brownian motion $B$ with respect to the filtration $\mathbb{F}$, which is a much weaker assumption.
Indeed, in that case, we just define in the above proof the process $Y$ by $Y_t:= B_t-B_T$ for $t\geq T$. The rest of the argument then works in the same way.
\end{remark}

Following F{\"o}llmer and Schachermayer \cite{MR2403770}, we now define the notion of $(\varepsilon_1,\varepsilon_2)$-arbitrage $($up to time $T)$.
\begin{definition}\label{e1e2}
Fix any $T<\infty$ and let $0<\varepsilon_1,\varepsilon_2<1$. The process $S$ \textit{admits an $(\varepsilon_1,\varepsilon_2)$-arbitrage up to time $T$} if there exists $X_T \in \mathbf{K}^T$ such that\\
$a)$ $X_T \geq -\varepsilon_2 \ \ P$-a.s.\\
$b)$ $P[X_T \geq 1-\varepsilon_2]\geq 1-\varepsilon_1$.
\end{definition}
Our next preliminary result is a direct consequence of Proposition 2.3 in F{\"o}llmer and Schachermayer \cite{MR2403770}.
More precisely, the result follows by the argument $(\mbox{ii}) \Rightarrow (\mbox{iii}) \Rightarrow (\mbox{i})$ in that proposition. See also Remark 2.4 in F{\"o}llmer and Schachermayer \cite{MR2403770}.
\begin{lemma}\label{trickaea}
Fix any $T<\infty$ and let $0<\varepsilon_1,\varepsilon_2<1$ be such that for
each $Q \in \mathbf{M}^{T,e}_m$, there is a set $A^Q_T \in \F_T$ with
$P[A^Q_T]\leq \varepsilon_1$ and $Q[A^Q_T]\geq 1-\varepsilon_2$. Then we have
for any $0<\tilde{\varepsilon}_1,\tilde{\varepsilon}_2<1$ with $2^{1+\alpha}
\max(\varepsilon_1,\varepsilon_2^{\alpha})\leq \tilde{\varepsilon}_1
\tilde{\varepsilon}_2^{\alpha}$ for some $0<\alpha<\infty$ that $S$ admits an
$(\tilde{\varepsilon}_1,\tilde{\varepsilon}_2)$-arbitrage up to time $T$.
\end{lemma}
Note that $\varepsilon_1, \varepsilon_2$ in the assumption of Lemma \ref{trickaea} are exogenously given and unrelated to $T$.
 The point of the next result is that it allows us to choose them both exponentially small in $T$, if $S$ satisfies the extra condition
of a large deviations estimate. This is the key for subsequently proving our main result.
\begin{proposition}\label{foschaconjL}
Under Assumption \ref{assAEAgenc}, suppose that a market price of risk
$\lambda$ for the price process $S$ satisfies a large deviations estimate.
Then there exist constants $\tilde{C}, \gamma_1, \gamma_2>0$ and $1\leq
T_0<\infty$ such that for all $T \geq T_0$, we can find for any $Q \in
\mathbf{M}^{T,e}_m$ a set $A^Q_T \in \F_T$ with
\begin{equation*}
P[A^Q_T]\leq \tilde{C} e^{-\gamma_1 T}<1 \ \ \ \  \textnormal{and}  \ \ \ \ Q[A^Q_T] \geq 1-e^{-\gamma_2 T} .
\end{equation*}
\end{proposition}
\begin{proof}
By assumption, there exist constants $c_1,c_2 >0$ such that as in (\ref{large6a}),
\begin{equation*}
\limsup_{T\to \infty} \frac{1}{T} \log P\bigg[ \frac{1}{T} \ K_T \leq c_1  \bigg]<-c_2 .
\end{equation*}
We take any constant $0<\delta<\frac{c_1}{2}$ and set
\begin{equation}\label{AEAkaivT1}
\gamma_1:= \min\bigg\{\frac{(c_1-2\delta)^2}{8c_1}, \frac{c_2}{2} \bigg \}>0 , \ \ \ \ \ \gamma_2:= \delta>0 , \ \ \ \ \
\tilde{C}:= \frac{\sqrt{2 c_1}}{(c_1-2\delta) \sqrt{\pi}} + 1 >0 .
\end{equation}
By the definition of $\limsup_{T\to \infty}$, we find $1\leq T_0<\infty$ such that for all $T \geq T_0$,
\begin{equation}\label{foschaconjLN1neu3}
\tilde{C} e^{-\gamma_1 T}<1 \ \ \ \ \mbox{and} \ \ \ \   
P\big[K_T \leq c_1 T \big] \leq e^{-\frac{c_2}{2} T} . 
\end{equation}
Fix $T \geq T_0$ and $Q \in \mathbf{M}^{T,e}_m$. For any stopping time
$\sigma\leq T$, Lemma \ref{AEAchardensgenc} gives that
\begin{equation}\label{foschaconjLN2}
\frac{dQ}{dP}\bigg|_{\F_\sigma}=Z^Q_\sigma=\exp\Big(L^Q_\sigma -\frac{1}{2} \big\langle L^Q \big\rangle_\sigma \Big) \ \ \ \mbox{with} \ \ \
\langle L^Q \rangle_\sigma \geq K_{\sigma} .
\end{equation}
We define the set $G_T^Q:=\{\langle L^Q \rangle_T > c_1 T \big\}$. Then (\ref{foschaconjLN1neu3}) and (\ref{foschaconjLN2}) imply that
\begin{equation}\label{AEAkaivT2}
 P\big[\big(G_T^Q\big)^{c}\big] = P\big[\big\langle L^Q \big\rangle_T \leq c_1 T\big]\leq P\big[ K_T \leq c_1 T \big] \leq e^{-\frac{c_2}{2} T} .
\end{equation}
Now, Lemma \ref{AEAkail1} yields a continuous local martingale $\bar{L}^Q=(\bar{L}^Q_t)_{t\geq 0}$ with $\bar{L}^Q = L^Q$ on $[\![0,T]\!]$ and $\langle \bar{L}^Q  \rangle_{\infty}=\infty$ $P$-a.s. We define the stopping times
\begin{equation*}
\tau^Q_t:= \inf\big\{s>0 \ \big| \ \langle \bar{L}^Q  \rangle_s > t \big\}
\end{equation*}
for any $t\geq 0$ and the process $B^Q=(B^Q_t)_{t\geq 0}$ by
\begin{equation} \label{AEAkaivT3}
B^Q_t:= \bar{L}^Q_{\tau^Q_t}.
\end{equation}
Then the Dambis--Dubins--Schwarz theorem (see Theorem 3.4.6 in \cite{KAR00}) implies that $B^Q$ is a Brownian motion.
Set
$\tau^Q:= \tau^Q_{c_1 T} \wedge T$ .
By definition, $\tau^Q$ is a stopping time with respect to $\mathbb{F}$ and values in $[0,T]$. Moreover, as
$\bar{L}^Q = L^Q$ on $[\![0,T]\!]$ and $\langle \bar{L}^Q \rangle$ is continuous, we obtain that
\begin{equation}\label{AEAkaivT4}
G_T^Q=\big\{ \big\langle \bar{L}^Q \big\rangle_T>c_1 T \big\} \subseteq \big\{ \tau^Q_{c_1 T}<T \big\} = \big\{ \tau^Q_{c_1 T} =\tau^Q<T \big\} .
\end{equation}
We also note that for any standard normal random variable $U$, we have the estimate
\begin{equation}\label{AEAkail2}
P[U> ab] \leq \frac{1}{\sqrt{2\pi}\, a} \ e^{-\frac{1}{2} a^2 b^2}
\end{equation}
for any $a>0$, $b\geq 1$.
For the set $A^Q_T:= \{Z^Q_{\tau^Q} > e^{-\delta T} \} \in \F_{\tau^Q}\subseteq \F_T$, (\ref{AEAkaivT4}), (\ref{foschaconjLN2}), Lemma \ref{AEAkail1}, (\ref {AEAkaivT3}), (\ref{AEAkail2}) and (\ref{AEAkaivT2}) then yield
\begin{align*}
P\big[A^Q_T\big]
&= P\Big[\big\{Z^Q_{\tau^Q}> e^{-\delta T} \big\} \cap G^Q_T \Big] + P\Big[\big\{Z_{\tau^Q}^Q> e^{-\delta T} \big\} \cap \big(G^Q_T\big)^{c} \Big]\\
&\leq P\Big[Z^Q_{\tau^Q_{c_1 T}}> e^{-\delta T}, \ \ \tau^Q_{c_1 T}=\tau^Q<T \Big]
+ P\Big[\big(G^Q_T\big)^{c} \Big]\\
&= P\Big[L^Q_{\tau^Q_{c_1 T}} - \frac{1}{2} \langle L^Q \rangle_{\tau^Q_{c_1 T}} > -\delta T, \ \ \tau^Q_{c_1 T}=\tau^Q<T \Big]
+ P\Big[\big(G^Q_T\big)^{c} \Big]\\
&= P\Big[\bar{L}^Q_{\tau^Q_{c_1 T}} - \frac{1}{2} \langle \bar{L}^Q \rangle_{\tau^Q_{c_1 T}} > -\delta T, \ \ \tau^Q_{c_1 T}=\tau^Q<T \Big]
+ P\Big[\big(G^Q_T\big)^{c} \Big]\\
&\leq P\Big[ B_{c_1 T}^Q -\frac{1}{2} c_1 T > -\delta T \Big] + P\Big[\big(G^Q_T \big)^{c} \Big]\\
&= P\Big[B_1^Q > \frac{c_1 -2\delta}{2\sqrt{c_1}} \sqrt{T}\Big] + P\Big[\big(G^Q_T\big)^{c} \Big]\\
&\leq \frac{2 \sqrt{c_1}}{(c_1 -2\delta) \sqrt{2\pi}} \
\exp \bigg(- \frac{(c_1-2\delta)^2}{8 c_1} T \bigg) + \exp\bigg(- \frac{c_2}{2} T \bigg).
\end{align*}
Combining this with (\ref{AEAkaivT1}) and (\ref{foschaconjLN1neu3}) gives
\begin{equation*}
P\big[A^Q_T\big] \leq \tilde{C} e^{-\gamma_1 T} < 1 .
\end{equation*}
Moreover, we deduce from the definition of $A^Q_T$ and as $\delta=\gamma_2$ that
\begin{equation*}
Q\big[A^Q_T\big]= 1-Q\big[\big(A^Q_T\big)^c\big]=1-E\big[Z^Q_{\tau^Q} 1_{(A^Q_T)^c}\big]\geq 1-e^{-\gamma_2 T}.
\end{equation*}
\end{proof}

Thanks to the quantitative strengthening achieved in Proposition \ref{foschaconjL}, we are now able to prove the announced result.
\begin{theorem}\label{foschaconjT}
Under Assumption \ref{assAEAgenc}, suppose that a market price of risk $\lambda$ for the price process $S$ satisfies a large deviations estimate. Then $S$  allows asymptotic exponential arbitrage with exponentially decaying failure probability.
\end{theorem}
\begin{proof}
By Proposition \ref{foschaconjL}, there exist $1\leq T_0<\infty$ and
constants $\tilde{C},\gamma_1, \gamma_2 >0$ such that for any $T \geq T_0$,
we can find for any $Q \in \mathbf{M}^{T,e}_m$ a set $A^Q_T \in \F_T$ with
\begin{equation*}
P[A^Q_T]\leq \tilde{C} e^{-\gamma_1 T}=: \varepsilon_{1,T} <1  \ \ \ \ \textnormal{and} \ \ \ \  Q[A^Q_T] \geq 1-e^{-\gamma_2 T}=:1-\varepsilon_{2,T} .
\end{equation*}
In particular, $\gamma_1 T>\log \tilde{C}$. For any $T \geq T_0$, we define
\begin{equation*}
\alpha_T:=\frac{\log \tilde{C}-\gamma_1 T}{-\gamma_2 T}>0 .
\end{equation*}
Thus $\alpha_T$ converges increasingly to $\frac{\gamma_1}{\gamma_2}$ as $T \to \infty$, and we have
\begin{equation}\label{foschaconjTN1}
\varepsilon_{1,T}=\varepsilon_{2,T}^{\alpha_T} .
\end{equation}
We take $T_0\leq \tilde{T}<\infty$ and a constant $\gamma_3$ with $0<2\gamma_3<\frac{\gamma_2}{2}$ and such that for any $T\geq \tilde{T}$,
\begin{equation}\label{foschaconjTN2}
e^{(\frac{\gamma_2}{2}-\gamma_3)T}-1 \geq e^{\gamma_3 T}\ \ \ \textnormal{and} \ \ \  2^{1+\frac{ \gamma_1}{\gamma_2}} \sqrt{\tilde{C} e^{-\gamma_1 T}}<1 .
\end{equation}
Now fix any $T\geq \tilde{T}$, set $\tilde{\varepsilon}_{1,T}:= 2^{1+\frac{\gamma_1}{\gamma_2}} \sqrt{\varepsilon_{1,T}}<1$ and $\tilde{\varepsilon}_{2,T}:=\sqrt{\varepsilon_{2,T}}<1$. By construction, due to (\ref{foschaconjTN1}), we have that
\begin{equation*}
\tilde{\varepsilon}_{1,T} \tilde{\varepsilon}_{2,T}^{\alpha_T}= 2^{1+\frac{\gamma_1}{\gamma_2}} \sqrt{\varepsilon_{1,T}} \sqrt{\varepsilon_{2,T}^{\alpha_T}}\geq 2^{1+\alpha_T} \sqrt{\varepsilon_{1,T}} \sqrt{\varepsilon_{2,T}^{\alpha_T}}=2^{1+\alpha_T} \max\big(\varepsilon_{1,T}, \varepsilon_{2,T}^{\alpha_T} \big) .
\end{equation*}
Therefore, we obtain from Lemma \ref{trickaea} that $S$ admits $(\tilde{\varepsilon}_{1,T}, \tilde{\varepsilon}_{2,T})$-arbitrage up to time $T$, which means that there is $\bar{X}_T \in \mathbf{K}^T$ such that\\
a) $\bar{X}_T \geq -e^{-\frac{\gamma_2}{2}T}$ \ $P$-a.s.\\
b) $P[\bar{X}_T \geq 1-e^{-\frac{\gamma_2}{2}T}]\geq 1-2^{1+\frac{\gamma_1}{\gamma_2}} \sqrt{\varepsilon_{1,T}}$ . \\ \\
We set $X_T:=e^{(\frac{\gamma_2}{2}-\gamma_3)T} \bar{X}_T \in \mathbf{K}^T$,
$\gamma_4:=\frac{\gamma_1}{2}>0$ and $C:=
2^{1+\frac{\gamma_1}{\gamma_2}}\sqrt{\tilde{C}}>0$. Due to the definition of
$X_T$ and (\ref{foschaconjTN2}), we obtain that
\begin{align*}
P[X_T \geq e^{\gamma_3 T}]&\geq P[X_T \geq e^{(\frac{\gamma_2}{2}-\gamma_3)T}-1]\\
&\geq P[X_T \geq e^{(\frac{\gamma_2}{2}-\gamma_3)T}-e^{-\gamma_3T}]\\
&= P[\bar{X}_T \geq 1-e^{-\frac{\gamma_2}{2}T}] .
\end{align*}
Thus, we conclude from the above properties of $\bar{X}_T$ and the definition of $\varepsilon_{1,T}$ that\\
a) $X_T \geq -e^{-\gamma_3 T}$ \ $P$-a.s.\\
b) $P[X_T \leq e^{\gamma_3 T}]\leq C e^{-\gamma_4 T}$ , \\
which proves the assertion.
\end{proof}
As a direct corollary, we can prove the conjecture in F{\"o}llmer and Schachermayer \cite{MR2403770}.
\begin{corollary}
Let $(\tilde{\Omega}, \tilde{\F},\tilde{\mathbb{F}},\tilde{P})$ be a filtered probability space
where the filtration $\tilde{\mathbb{F}}=(\tilde{\F}_t)_{t\geq 0}$ is the $\tilde{P}$-augmentation of the raw filtration generated by an
$\R^N$-valued Brownian motion $\tilde{W}$. Moreover, let $\tilde{S}$ be the diffusion process defined in (\ref{foschaN6}). Suppose that
the market price of risk function $\varphi(\cdot)$ satisfies a large deviations estimate with respect to $\tilde{S}$. Then $\tilde{S}$
allows asymptotic exponential arbitrage with exponentially decaying failure probability.
\end{corollary}
\begin{proof}
By our assumption (\ref{AEAfoschamartass}) on the diffusion process in (\ref{foschaN6}) and the choice
 of the filtration $\tilde{\mathbb{F}}$, the martingale representation theorem implies
 Assumption \ref{assAEAgenc} for $\tilde{P}$. Moreover, it is well known that
for every $T<\infty$ and any equivalent martingale measure $\tilde{Q}$ for $(\tilde{S}_t)_{0\leq t \leq T}$, the density process $\tilde{Z}^{\tilde{Q}}$ with respect to $\tilde{P}|_{\tilde{\F}_T}$ is of the form
\begin{equation*}
 \tilde{Z}^{\tilde{Q}}= \mathcal{E}\bigg(\int -\psi^{\tilde{Q}}\, d\tilde{W} \bigg)=:
 \mathcal{E}\Big(\tilde{L}^Q\Big) ,
\end{equation*}
 where $(\psi_t^{\tilde{Q}})_{0\leq t \leq T}$ is a predictable $\R^N$-valued process with $\psi_t^{\tilde{Q}}-\varphi(\tilde{S}_t) \in$
 ker$(\sigma(\tilde{S}_t))$ for any $t \in [0,T]$. As a consequence, we have
\begin{equation*}
\big\langle \tilde{L}^Q \big \rangle_t= \int_0^t \Vert \psi^{\tilde{Q}}_s \Vert^2\, ds \geq \int_0^t \Vert \varphi(\tilde{S}_s) \Vert^2\, ds
\end{equation*}
for any $t \in [0,T]$. For details, we refer to Section 3 of F{\"o}llmer and Schachermayer \cite{MR2403770}.
 Therefore, if we compare Definitions \ref{large6} and \ref{large6diff} and look at Lemma \ref{AEAchardensgenc},
 we see that we get the result directly by using the same computations as in Proposition \ref{foschaconjL} and Theorem \ref{foschaconjT}, replacing $L^Q$ by $\tilde{L}^Q$ and $K$ by $\int \Vert \varphi(\tilde{S}_s) \Vert^2 ds$.
\end{proof}
In F{\"o}llmer and Schachermayer \cite{MR2403770}, the authors considered the diffusion process $\tilde{S}$ defined in (\ref{foschaN6}) and
introduced a quantitative form of long-term arbitrage. This is almost the same as asymptotic exponential arbitrage with
exponentially decaying failure probability, with the difference that there is no relation between
the tolerance level of failure and the necessary time to reach a level.
The authors introduced the notion of \textit{having an average squared market price of risk above a threshold $c>0$},
which is a growth condition on the
mean-variance tradeoff process. They proved that if $\tilde{S}$ satisfies this, then there exists the above kind of
long-term arbitrage (see Theorem 1.4 in \cite{MR2403770}). Furthermore, the authors wrote that one should expect to have asymptotic
exponential arbitrage with exponentially decaying failure probability (in the sense of the present paper)
 under the stronger assumption that the market price of risk function $\varphi(\cdot)$ for $\tilde{S}$
satisfies a large deviations estimate. They even sketched an argument how one could try to prove this conjecture using a
 large deviations approach, but left the details and precise assumptions open.
In Mbele Bidima and R\'asonyi \cite{MbeleR}, the authors proved such a result in a discrete-time version of the model (\ref{foschaN6})
by using a large deviations estimate for a martingale difference sequence (see Theorem 4 in \cite{MbeleR}).
The main contribution of the present paper is a rigorous proof based on a time-change argument instead of a large deviations approach.
In addition to avoiding any extra assumptions, this has also allowed us to prove the result not only for diffusions, but for general
continuous semimartingales (satisfying Assumption \ref{assAEAgenc}).

\section*{Acknowledgements}
The authors wish to thank Martin Schweizer for his careful reading, remarks
and suggestions which have improved the presentation of this paper. Financial
support by the National Centre of Competence in Research "Financial Valuation
and Risk Management" (NCCR FINRISK), Project D1 (Mathematical Methods in
Financial Risk Management) is gratefully acknowledged. The NCCR FINRISK is a
research instrument of the Swiss National Science Foundation.

\addcontentsline{toc}{section}{References}

\providecommand{\bysame}{\leavevmode\hbox to3em{\hrulefill}\thinspace}
\providecommand{\MR}{\relax\ifhmode\unskip\space\fi MR }
\providecommand{\MRhref}[2]{%
  \href{http://www.ams.org/mathscinet-getitem?mr=#1}{#2}
} \providecommand{\href}[2]{#2}

\end{document}